\documentclass[11pt]{article}
\usepackage[english]{babel}
\usepackage{amsfonts, amsmath, amssymb,amscd}
\usepackage[all]{xy}
\usepackage{latexsym}
\usepackage{amsmath, amsthm}
\usepackage{amsfonts}
\usepackage{amssymb}
\usepackage{amsmath}
\usepackage{graphicx,color}

\usepackage{tikz}
\usepackage{verbatim}

\newtheorem{theorem}{Theorem}[section]
\newtheorem{lemma}[theorem]{Lemma}

\newtheorem{corollary}[theorem]{Corollary}

\newcommand{\Aop}{\mathcal{A}}

\begin{document}

\title{On the nature of isolated asymptotic singularities  of solutions of a family of quasi-linear elliptic PDE's on a  Cartan-Hadamard manifold }
\author{Leonardo Bonorino
\and Jaime Ripoll}

\date{}
\maketitle

\begin{abstract}
Let $M$ be a  Cartan-Hadamard manifold with sectional curvature satisfying $-b^2\leq K\leq -a^2<0$, $b\geq a>0.$ Denote by $\partial_{\infty}M$ the asymptotic boundary of $M$ and by $\bar M:= M\cup\partial_\infty M$ the  geometric compactification of $M$ with the cone topology. We investigate here the following question: Given a finite number of points $p_{1},...,p_{k}\in\partial_\infty M,$
 if $u\in C^{\infty}(M)\cap C^{0}\left(
\bar{M}\backslash\left\{  p_{1},...,p_{k}\right\}  \right)  $ satisfies a PDE $\mathcal Q(u)=0$ in $M$ and if $u|_{\partial_\infty M\setminus\left\{p_{1},...,p_{k}\right\}}$
extends continuously to $p_{i},$ $i=1,...,k,$ can one conclude that $u\in C^{0}\left(
\bar{M}\right )?$ When $\dim M=2$, for $\mathcal Q$ belonging to a  linearly convex space of quasi-linear elliptic operators $\mathcal{S}$ of the form
$$
\mathcal{Q}(u)=\operatorname{div}\left( \frac{\Aop(|\nabla u|)}{|\nabla u|} \nabla u \right)=0,  
$$
where $\Aop$ satisfies some structural conditions, then the answer is yes provided that $\Aop$ has a certain asymptotic growth.  This condition includes, besides the minimal graph PDE,  a  class of minimal type PDEs.
 
In the hyperbolic space $ \mathbb{H}^n$, $n\ge 2,$ we are able to give a complete answer: we prove that $\mathcal{S}$  splits into two disjoint classes of minimal type and $p-$Laplacian type PDEs, $p>1,$ where the answer is yes and no respectively. These two classes are determined by the  asymptotic behaviour of $\mathcal A.$  Regarding the class where the answer is negative, we obtain  explicit  solutions having an isolated non removable singularity at infinity.
\end{abstract}

\section{Introduction}

\qquad Let $M$ be Cartan-Hadamard $n-$dimensional manifold (complete, connected, simply
connected Riemannian manifold with non-positive sectional curvature). It is well-known
that  $M$ can be compactified with the so called
cone topology by adding a sphere at infinity, also called the asymptotic
boundary of $M$; we refer to \cite{EO} for details. In the sequel, we will
denote by $\partial _{\infty }M$ the sphere at infinity and by $\bar{M}%
=M\cup \partial _{\infty }M$ the compactification of $M$.

We recall that the asymptotic Dirichlet problem of a PDE $\mathcal Q(u)=0$ in $M$ for a
given asymptotic boundary data $\psi \in C^{0}\left( \partial _{\infty
}M\right) $ consists in finding a solution $u\in C^{0}\left( \bar{M}%
\right) $ of $Q(u)=0$ in $M$ such that $u|_{\partial _{\infty }M}=\psi $,
determining the uniqueness of $u$ as well$.$ 

The asymptotic Dirichlet
problem for the Laplacian PDE has been studied during the last 30 years and
there is a vast literature in this case. More recently, it has been studied
in a larger class of PDEs  which include the $p-$Laplacian PDE, $p>1,$ 

$$
\Delta_p u=\operatorname{div}\frac{\nabla u}{\left\vert \nabla
u\right\vert ^{p}}=0,  
$$
see \cite{H}, and the minimal graph PDE, 

\begin{equation}
\mathcal{M}(u)=\operatorname{div}\frac{\nabla u}{\sqrt{1+\left\vert \nabla
u\right\vert ^{2}}}=0,  \label{min}
\end{equation}%
see \cite {GR}, \cite {RT1}, case that we are specially interested in the present work. We note that $\operatorname{div}$
and $\nabla $ are the divergence and the gradient in $M$ and it is worth to mention that the graph 
$$G(r)=\left\{  (  x,u(x))  \text{
$\vert \quad$} x\in M\right\}  $$
of $u$ is a minimal surface in $M\times \mathbb{R}$ if and only if $u$
satisfies (\ref{min}). 

Presently it is known that the asymptotic Dirichelt problem can be solved in any  Cartan-Hadamard manifold under hypothesis on the growth of  the sectional curvature that includes the ones with negatively pinched curvature, for any given continuous data  at infinity, and  on a large class of PDEs that includes both $p-$Laplacian and minimal graph PDEs (see \cite{CHR}, \cite
{RT2}).

A natural question related
to the asymptotic Dirichlet problem concerns the existence
or not of solutions with isolated singularities at $\partial _{\infty }M.$ We investigate  this problem on the following class  $\mathcal{S}$ of quasi-linear elliptic operators:

\begin{equation}
\mathcal{Q}(u)=\operatorname{div}\left( \frac{\Aop(|\nabla u|)}{|\nabla u|} \nabla u \right)=0,  \label{generalOperator}
\end{equation}
where   $\Aop \in C^1[0,\infty)$ satisfies the following conditions:
\begin{equation}
\left. \begin{array}{l} 
\Aop(0)=0, \Aop'(s) > 0 \; {\rm for} \; s >0; \\[2pt]
\Aop(s) \le C(s^{p-1} + 1) \; {\rm  for \; some} \; C>0, \; {\rm some}  \; p \ge 1 \; {\rm and \; any} \; s >0; \\[2pt]
{\rm there \; exist \; positives} \; q {\rm ,} \; \delta_0 \; {\rm and } \; \bar{D} \; {\rm s.t.} \; \Aop(s) > \bar{D} s^q \; {\rm for} \; s \in [0,\delta_0].
\end{array} \right\} 
\label{basicHypothesesOna(s)}
\end{equation}

This class of operators, as the authors know, was first introduced and studied regarding the solvability of the asymptotic Dirichlet problem  in \cite{RT2}; it    includes well known geometric operators as
the $p$-laplacian, for $p >1$, ($\Aop(s)=s^{p-1}$) and the minimal graph operator  ($\Aop(s)=s/\sqrt{1+s^2}$).  Note that $\mathcal{S}$  is linearly convex that is, any two elements $\mathcal Q_1, \mathcal Q_2$ of $\mathcal{S}$ are homothopic in $\mathcal{S}$ by the line segment $t\mathcal Q_1+(1-t)\mathcal Q_2$, $0\le t\le 1.$   

As we shall see, the nature of an isolated asymptotic singularity of  $\mathcal{Q}$ depends on the asymptotic behaviour of $\Aop$ and can change drastically accordingly to it. It is worth to mention at this point that  this behaviour of $\Aop$ is closely related to the existence or not of ``Scherk type'' solutions of \eqref{generalOperator} (see the beginning of the next section). Minimal  Scherk surfaces play a fundamental role on the theory of minimal surfaces in Riemannian manifolds (a well known breakthrough result using Scherk minimal surfaces  were obtained by P. Collin and H. Rosenberg in \cite{CR}). 

\indent In our first three results we are concerned with removable singularities. We first show that isolated singularities are  removable  if $n=2, $ $M$ has negatively pinched curvature and $\Aop$  satisfies
$$
\int_0^{\infty} \Aop^{-1}(K_0 (\cosh(a r))^{-1}) \, dr = +\infty,
$$
for some $K_0 >0 $. Since $\Aop^{-1}(t) \le ct^{1/q}$ holds for small $t$, due to \eqref{basicHypothesesOna(s)}, the change of variable $t=K_0 (\cosh(a r))^{-1}$ implies that this condition is equivalent to
\begin{equation}
\int_0^{K_0} \frac{\Aop^{-1}(t)}{ \sqrt{K_0 - t}} \, dt = +\infty.
\label{IntegralCondition1Ona(t)}
\end{equation} 
Precisely, we prove:

\begin{theorem}
\label{fi}
Suppose that $M$ is a $2-$dimensional Cartan-Hadamard manifold with sectional curvature satisfying $-b^{2}\leq K\leq -a^{2}<0$, $b\geq a>0.$
 Given a finite number of points $p_{1},...,p_{k}\in
\partial _{\infty }M,$ if $m\in C^{\infty }(M)\cap  C^{0}\left( \bar{M}%
\backslash \left\{ p_{1},...,p_{k}\right\} \right) $ is a solution of
\eqref{generalOperator} in $M$, $\Aop(s)$ satisfies \eqref{basicHypothesesOna(s)} and \eqref{IntegralCondition1Ona(t)}, and $m|_{\partial _{\infty }M\setminus
\left\{ p_{1},...,p_{k}\right\} }$ extends continuously to $p_{i},$ $%
i=1,...,k,$ then $m\in C^{0}\left( \bar{M}\right) .$
\end{theorem}

We observe that condition \eqref{IntegralCondition1Ona(t)} fails if $K_0 < \sup \Aop$.
Hence, \eqref{IntegralCondition1Ona(t)} implies that $\Aop$ is bounded and $K_0=\sup \Aop$. This happens, for instance, if $\Aop(s)=s/\sqrt{1+s^2}$. Therefore, we have
\begin{corollary}
Suppose that $M$ is a $2-$dimensional Cartan-Hadamard manifold with sectional curvature satisfying $-b^{2}\leq K\leq -a^{2}<0$, $b\geq a>0.$
 Given a finite number of points $p_{1},...,p_{k}\in
\partial _{\infty }M,$ if $m\in C^{\infty }(M)\cap  C^{0}\left( \bar{M}%
\backslash \left\{ p_{1},...,p_{k}\right\} \right) $ is a solution of
the minimal surface equation and if $m|_{\partial _{\infty }M\setminus
\left\{ p_{1},...,p_{k}\right\} }$ extends continuously to $p_{i},$ $%
i=1,...,k,$ then $m\in C^{0}\left( \bar{M}\right) .$
\label{fiminimal}
\end{corollary}  

We observe that a similar problem can obviously be posed to solutions of \eqref{generalOperator} on a bounded $C^{0}$ domain $\Omega $ of $%
\mathbb{R}^{2}.$ In the minimal case, this a an old problem. From a classical result of R. Finn \cite{F},  it follows that if $u, $ as in the above theorem,  with $M$ replaced by $\Omega,$ $\partial_\infty$ by $\partial,$ is a solution of the minimal graph equation \eqref{min}  and if there there is a solution $v\in C^{\infty }(\Omega)\cap  C^{0}\left( \bar{\Omega}%
 \right)$ of \eqref{min}  such that $$u|_{\partial\Omega\setminus
\left\{ p_{1},...,p_{n}\right\} }=v|_{\partial\Omega\setminus
\left\{ p_{1},...,p_{n}\right\} }$$ 
then $u=v$ and hence $u$ extends continuously through the singularities.
If the Dirichlet problem $\mathcal{M}(u)=0$ on $\Omega $ is not solvable for
the continuous boundary data $\phi:=u|_{\partial\Omega}$  then the result is false, a known fact on the  classical minimal surface theory (see \cite{N}, Chapter V, Section 3).  We remark that even if the Dirichlet problem is not solvable there might exist smooth compact minimal surfaces which boundary  is the graph of $\phi$ if $\phi$ and the domain are regular enough  (see \cite{B}).

Although under the hypothesis of Corollary \ref{fiminimal} there exists a solution 
$v\in C^{\infty }(M)\cap  C^{0}\left( \bar{M} \right) $ of (\ref{min}) such that
$u|_{\partial_{\infty} M\setminus
\left\{ p_{1},...,p_{n}\right\} }=v|_{\partial_{\infty}M\setminus
\left\{ p_{1},...,p_{n}\right\} },$
we felt necessary to use a different approach from Finn's. First because  the boundedness of the domain is fundamental to the arguments used in \cite{F}.  Secondly, because  it is not clear that the asymptotic Dirichlet problem for the PDE (\ref{generalOperator}), under the conditions 
(\ref{basicHypothesesOna(s)}),
 is solvable for any continuous  boundary data given at infinity.

Our proof relies  heavily on asymptotic properties of $2-$dimensional Cartan-Hadamard manifolds. It is fundamentally based on  the fact that a point $p$ of the asymptotic boundary of $M$ is an isolated point of the asymptotic boundary of a domain $U$ such that $M\setminus U$ is convex. This property allows the construction of suitable barriers at infinity. Although the existence of $U$ in the $n=2$ dimensional case is trivial (for example, a domain which boundary are two geodesics asymptotic to $p$), we don't know if such an $U$ exists in $M$  if $n\geq 3$. Nevertheless,  it is possible in the special case of the hyperbolic space  to give an ad hoc proof of Theorem \ref{fi} using the symmetries of the space. Precisely, our result in $\mathbb{H}^n$ reads:
 
 \begin{theorem}
\label{sec}
Let $\mathbb{H}^n$ be the hyperbolic space of constant section curvature $-1$.  Given a finite number of points $p_{1},...,p_{k}\in \partial_{\infty} \mathbb{H}^n$,
if $m\in C^{\infty }(\mathbb{H}^n)\cap  C^{0}\left( \bar{\mathbb{H}^n}%
\backslash \left\{ p_{1},...,p_{k}\right\} \right) $ is a solution of \eqref{generalOperator} in $\mathbb{H}^n$, $\Aop(s)$ satisfies \eqref{basicHypothesesOna(s)} and \eqref{IntegralCondition1Ona(t)}, and if $m|_{\partial _{\infty }\mathbb{H}^n\setminus
\left\{ p_{1},...,p_{k}\right\} }$ extends continuously to $p_{i},$ $%
i=1,...,k,$ then $m\in C^{0}\left( \bar{\mathbb{H}^n}\right) .$
\end{theorem}

Finally, in the next last result, we prove the existence of a class of solutions of \eqref{generalOperator} in $\mathbb{H}^n$  admiting a non removable isolated asymptotic singularity. Note that this class contains the $p-$Laplacian PDE, $p>1.$
\begin{theorem}
Suppose that \eqref{basicHypothesesOna(s)} holds and $\Aop(s)$ is unbounded.  Given a point $p_{1}\in \partial_{\infty} \mathbb{H}^n$,
there exists a solution $m\in C^{\infty }(\mathbb{H}^n)\cap  C^{0}\left( \bar{\mathbb{H}^n}
\backslash \left\{ p_{1} \right\} \right) $  of \eqref{generalOperator} in $\mathbb{H}^n$, such that $m=0$ on $\partial_{\infty} \mathbb{H}^n \backslash \{p_1\}$ and 
$\limsup_{x \to p_1} m = +\infty.$
\label{nonexistenceThm}
\end{theorem}

\section{Proof of the theorems}
We begin by constructing Scherk type supersolutions to the equation \eqref{generalOperator}, which are fundamental to prove the nonexistence of true asymptotic singularities. 
\begin{lemma}
Let $\gamma$ be some geodesic of $M$, let $U$ be one of the connected component of $M\backslash \gamma$ and $\delta > 0$. If $\Aop$ satisfies \eqref{basicHypothesesOna(s)} and \eqref{IntegralCondition1Ona(t)}, then there exists a solution of 
$$ \left\{ \begin{array}{rcl}   {\rm div} \left(\displaystyle \frac{\Aop (|\nabla u|)}{|\nabla u |} \nabla u \right) & \leq & 0 \quad {\rm in } \quad U \\[5pt]
                                                                                u   & = & +\infty \quad {\rm on } \quad  \gamma \\[5pt]
                                                                                u   & = & \delta \quad {\rm in } \quad {\rm int} \; \partial_{\infty} U. \\ \end{array} \right. $$
\label{existenceOfSomeScherk}
\end{lemma}
\begin{proof} Let $d:U \to \mathbb{R}$ be defined by $d(x)= dist(x,\gamma)$ and $g: (0,+\infty) \to \mathbb{R}$ be defined by
$$ g(d) = \delta + \int_d^{\infty} \Aop^{-1} \left( \frac{K_0 }{\cosh (a t )}\right) \; dt, $$
where $K_0 = \sup \Aop$. Observe that according to \cite{RT2}, $g(d)$ is well defined and finite for all $d > 0$, and $v(x):= g(d(x))$ is a supersolution of \eqref{generalOperator}. Moreover, $g(d) \to \delta$ as $d \to +\infty$ and, therefore, $g(d(x)) \to \delta$ as $x \to p \in \partial_{\infty} U$. That is, $v = \delta$ on ${\rm int} \; \partial_{\infty} U.$ 
Finally, making the change of variable $z=K_0 (\cosh (a t ))^{-1} $, we can prove that condition \eqref{IntegralCondition1Ona(t)} implies that  $g(d) \to + \infty$ as $d \to 0$. Hence $v(x)=g(d(x)) \to +\infty$ as $x \to x_0 \in \gamma$, completing the lemma.
\end{proof}

\subsection{Proof of Theorem \ref{fi}}
\qquad We first claim that $m$ is bounded: 
For each $p_i$, consider a geodesic $\Gamma_i$ such that the asymptotic boundary of one of the connected components of $M\setminus\Gamma_i$, say $X_i$, does not contain $p_j$ for $j \ne i$. Assume also that $p_i \in {\rm int}\; \partial_{\infty} X_i$. Since $\Gamma_i(\pm \infty) \not\in \{p_1, \dots \, p_n\}$, $m$ is continuous at $\Gamma_i(\pm \infty)$ and therefore it is bounded on $\Gamma_i$. Let $ S_i = \displaystyle \sup_{\Gamma_i}m$ for $i \in \{1,\dots n\}$, $S_0 = \sup m|_{\partial _{\infty }M\setminus
\left\{ p_{1},...,p_{n}\right\} } $ and $$ S=\max\{S_0, S_1, \dots, S_n \}.$$ 
From the maximum principle, $m \le S$ in $M\backslash \{X_1 \cup \dots \cup X_n \}$. To prove that $m \le S$ in $X_i$, take a sequence of geodesics $\beta_k$ such that the ending points $\beta_k(+\infty)$ and $\beta_k(-\infty)$ converge to $p_i$. Let $Y_k$ be the connected component of $M\backslash \beta_k$ whose the asymptotic boundary does not contain $p_i$. Observe that $M \backslash X_i \subset Y_k$ for large $k$ and $\cup Y_k = M$.  Let $w_k$ be the supersolution of (\ref{generalOperator}) given by Lemma \ref{existenceOfSomeScherk}. Recall that $w_k$
is $+\infty$ on $\beta_k$ and $S$ at $\partial_{\infty}Y_k \setminus\{\beta_k(\pm\infty)\}$. Hence $w_k \ge S$ and therefore $w_k \ge m$ on $\Gamma_i =\partial X_i$, $w_k = S \ge m$ on $\partial_{\infty}(X_i \cap Y_k)$ and $w_k=+\infty > m$ on $\beta_k = \partial Y_k$. Then $w_k \ge m$ in $Y_k \cap X_i$ for large $k$ by the Comparison Principle. For any given $x\in M$, $x \in Y_k$ for
large $k$. Hence, using that $w_k(x) \to S$, we have $m(x)\le S$. In a similar way, we can conclude that $m$ is bounded from below, proving the claim.  

Assume that $m\leq S$. Denote by $\phi$ the continuous extension of \\ $m|_{\partial _{\infty }M\setminus\left\{ p_{1},...,p_{n}\right\} }$ to $\partial _{\infty }M.$
Let $p\in\{p_{1},...,p_{n}\}$.  Adding a constant to $\phi$ we may assume  wlg that $\phi(p)=0.$ Let $0<\delta\leq S$ be given.
We will prove that   $K:=\limsup_{x \to p} m(x)\leq \delta.$   By contradiction assume that that $K > \delta.$

  By the continuity of $\phi,$ there exists an open connected neighborhood $\mathcal{O}\subset \partial_{\infty}M$ of $p$ such that $\phi(q)\le\delta$ for all $q\in \mathcal{O}.$
Moreover, we may assume that  $\mathcal{O}$ does not contain another point $p_{i}$ except
$p.$

Let $\gamma$ be a geodesic such that $\gamma(\infty)=p.$ Set $\gamma=\gamma(\mathbb{R})$. Choose a point $q_0\in \gamma$ and a geodesic $\alpha_0$ orthogonal to $\gamma$ at $q_0$ such that $\alpha_0(\pm\infty)\in \mathcal{O}.$ 
Let $\gamma_i$, $i\in\{1,2\}$, be the geodesics with ending points at $p$ and $q_1:=\alpha_0(\infty)$ and  $p$ and $q_2:=\alpha_0(-\infty),$ respectively. Denote by $U_i$ the connected component of $M\setminus\gamma_i$ that does not contain $\alpha_0.$
As before, there exists $Sh_i$ solution of 
 $$ \left\{ \begin{array}{rcl}   {\rm div} \left(\displaystyle \frac{\Aop (|\nabla u|)}{|\nabla u |} \nabla u \right) & \leq & 0 \quad {\rm in } \quad U_i \\[5pt]
                                                                                u   & = & +\infty \quad {\rm on } \quad  \gamma_i \\[5pt]
                                                                                u   & = & \delta \quad {\rm in } \quad {\rm int} \; \partial_{\infty} U_i. \\ \end{array} \right. $$

\                                                                               
                                                                                
\noindent Observe that $m < Sh_i$. Let $c_i$ be the level set of $Sh_i$
$$c_i = \left\{ x \in M \; : \; Sh_i(x) = \frac{K}{2}+\frac{\delta}{2} \right\}$$
and
$$ V_i = \left\{ x \in  U_i \; : \; Sh_i(x) < \frac{K}{2}+\frac{\delta}{2} \right\}$$
Hence $m < K/2+\delta/2$ on $V_i.$ 
Let $V = A \backslash (V_1 \cup V_2)$.

\

Now, let $W$ be a neighborhood of $p$ (a ball centered at $p$) such that the asymptotic boundary of $W \cap V$ is $\{p\}$. Observe that for  $R>0$ and any point $z$ on the boundary of $W \cap V$ there exist a ball of radius $R$, $B_R \subset M \backslash (W \cap V)$
such that $B_R \cap \overline{W\cap V} = \{z\}.$  We consider $R=1$.

Since $p$ is an ending point of both $\gamma_1$ and $\gamma_2$, the distance between any point of $W \cap V$ and the geodesic $\gamma_i$ is bounded by some constant.
This property still holds if we consider the curve $c_i$ instead $\gamma_i$, since these two curves are equidistant. Then there is $\rho > 0$ be such that 
$$ \operatorname*{dist}(x, V_i) < \rho \quad \quad {\rm for \; any } \; x \in W \cap V. $$
That is, for any $x \in W \cap V$, there is a ball $B_{\rho}$ centered at some point of $\partial (V_1 \cup V_2) \cap W$
s.t. $x \in B_{\rho}$.

\

\begin{minipage}{6.5cm}
\begin{tikzpicture}[scale=0.8,domain=0:4]
\draw (0,0) circle  (3);

\draw(0,3)--(0,-3);
\foreach \y in {6}
{
\pgfmathatan{3/(2*\y - 3)};
\draw[domain=315-\pgfmathresult:315+\pgfmathresult, variable=\t, color=black] plot ({sqrt(2*\y*\y-6*\y+9)*cos(\t) - \y},{sqrt(2*\y*\y-6*\y+9)*sin(\t) + \y});
}

\foreach \y in {6}
{
\pgfmathatan{3/(2*\y - 3)};
\draw[domain=45-\pgfmathresult:45+\pgfmathresult, variable=\t, color=black] plot ({sqrt(2*\y*\y-6*\y+9)*cos(\t) - \y},{sqrt(2*\y*\y-6*\y+9)*sin(\t) - \y});
}

\node at (-3.3, 0) {$p$};
\node at (-0.7, 0) {$V$};
\node at (0, 3.3) {$q_1$};
\node at (0, -3.3) {$q_2$};

\node at (-1.5, 1.8) {$V_1$};
\node at (-1.5, -1.8) {$V_2$};
\node at (-0.7, 1.2) {$c_1$};
\node at (-0.7, -1.2) {$c_2$};
\end{tikzpicture}
\end{minipage}
\begin{minipage}{6.5cm}
\begin{tikzpicture}[scale=0.8, domain=0:4]
\draw (0,0) circle  (3);

\foreach \y in {6}
{
\pgfmathatan{3/(2*\y - 3)};
\draw[domain=315-\pgfmathresult:315+\pgfmathresult/3, variable=\t, color=black] plot ({sqrt(2*\y*\y-6*\y+9)*cos(\t) - \y},{sqrt(2*\y*\y-6*\y+9)*sin(\t) + \y});
}

\foreach \y in {6}
{
\pgfmathatan{3/(2*\y - 3)};
\draw[domain=45-\pgfmathresult/3:45+\pgfmathresult, variable=\t, color=black] plot ({sqrt(2*\y*\y-6*\y+9)*cos(\t) - \y},{sqrt(2*\y*\y-6*\y+9)*sin(\t) - \y});
}

\foreach \y in {6}
{
\pgfmathatan{3/(2*\y - 3)};
\draw[domain=0:360, variable=\t, color=black] plot ({0.8*cos(\t) + (sqrt(2*\y*\y-6*\y+9)-0)*cos(45+\pgfmathresult/2)-\y},{0.8*sin(\t) + (sqrt(2*\y*\y-6*\y+9)-0)*sin(45+\pgfmathresult/2)- \y});
\draw[domain=0:360, variable=\t, color=black, fill=black] plot ({0.05*cos(\t) + (sqrt(2*\y*\y-6*\y+9)-0)*cos(45+\pgfmathresult/2)-\y},{0.05*sin(\t) + (sqrt(2*\y*\y-6*\y+9)-0)*sin(45+\pgfmathresult/2)- \y});
\draw[domain=0:360, variable=\t, color=black, fill=black] plot ({0.05*cos(\t) + 0.2+(sqrt(2*\y*\y-6*\y+9)-0)*cos(45+\pgfmathresult/2)-\y},{0.05*sin(\t) + 0.4+ (sqrt(2*\y*\y-6*\y+9)-0)*sin(45+\pgfmathresult/2)- \y});
}

\foreach \y in {40}
{
\draw[domain=\y-90:90-\y, variable=\t, color=black] plot ({3*tan(\y)*cos(\t) - 3/cos(\y)},{3*tan(\y)*sin(\t)});
}

\node at (-3.3, 0) {$p$};

\node at (-0.5, 2.0) {$c_1$};
\node at (-0.5, -2.0) {$c_2$};
\node at (-0.8, -0.3) {\textcolor{black}{$B_{\rho}$}};
\node at (-2.1, -0.15) {$x$};
\node at (-1.2, 0.7) {$W$};

\end{tikzpicture}
\end{minipage}
\\ \\
\null \hspace{2.8cm} Fig. 1 \hspace{5.5cm} Fig. 2

\

\
\begin{lemma}
\label{lem}
There exist $h_0$  and $h_1$  depending only on $b$, $\rho$, $K$ and $\delta$, satisfying $$\delta<h_1<h_0<K/2+\frac{\delta}{2}$$ such that, for any  $y\in M,$
the Dirichlet problem in the annulus $B_{2\rho + 1}(y) \backslash \overline{B_1(y)}$

$$ \left\{ \begin{array}{rcl}   {\rm div} \left( \displaystyle \frac{\Aop (|\nabla u|)}{|\nabla u |} \nabla u \right) & = & 0 \quad {\rm in } \quad B_{2\rho + 1}(y) \backslash \overline{B_1(y)}\\[5pt]
                                                                                u   & = & \delta \quad {\rm on } \quad \partial B_1(y) \\[5pt]
                                                                                u   & = & h_0 \quad {\rm on } \quad \partial B_{2\rho + 1}(y) \\ \end{array} \right. $$
has a supersolution $w_y(x) $ and  $ w_y(x) \le h_1$  if   $\operatorname*{dist}(x,y) < \rho +1.$
\end{lemma}

\begin{proof}

Let $f:[1,\infty)\to \mathbb{R}$ be the function defined by 
$$ f(r) = \delta + \int_1^r \Aop^{-1}\left( \frac{\sinh b \,\alpha}{\sinh(bs)} \right)\; ds, $$
where $0 < \alpha \le 1$. Hence $f(1)=\delta$ and, choosing $\alpha$ sufficiently small, $f(2\rho +1) < K/2 + \delta/2$. Let $h_0=f(2\rho +1)$.
Observe that if $r=r(\tilde{x})$ is the distance in $\mathbb{H}^2(-b^2)$ from $\tilde{x}$ to a fixed point, then the the graphic of $f$ is a radially symmetric surface, solution of \eqref{generalOperator}  in the hyperbolic plane with constant negative sectional curvature $-b^2$, that is,
$f$ satisfies 
$$ \Aop'(f'(r))f''(r)+ \Aop(f'(r))b\coth b r = 0.$$
Moreover, from the Comparison Laplacian Theorem $$\Delta d(x) \le \Delta r(\tilde{x})=b\coth b r,$$ where $d(x)=dist(x,y)$ and $\tilde{x} \in \mathbb{H}^2(-b^2)$ is a point such that $d(x)=r(\tilde{x}).$  Then, using these two relations and that $f' > 0$, we conclude that $w_y(x):= f(d(x))$ is a supersolution of \eqref{generalOperator} in $M.$
\\
\null \hspace{1em} Since $f(1)=\delta$ and $f(2\rho+1)=h_0$, $w_y(x)$ satisfies the required boundary conditions. Finally defining $h_1:= f(\rho+1)$, $w_y(x) \le h_1 < h_0$ in $B_{\rho+1}(y)$. 
\end{proof}

\

\noindent Let $\varepsilon$ be a positive real satisfying $h_0-h_1 - (K-\delta)/2 \le \varepsilon < h_0 - h_1$ and $W_0 \subset W$ be a neighborhood of $p$ (a ball centered at $p$) s.t.
$$ m < K +\varepsilon \quad {\rm in } \quad W_0. $$
Let $\tilde{W} \subset W_0$ be a neighborhood of $p$ (a ball centered at $p$) s.t.
$$ \operatorname*{dist}(\partial W_0, \tilde{W}) > 3\rho + 2.$$

\

\begin{center}
\begin{tikzpicture}[domain=0:4]
\draw (0,0) circle  (3);

\foreach \y in {25}
{
\draw[domain=\y-90:90-\y, variable=\t, color=black, fill=black!20] plot ({3*tan(\y)*cos(\t) - 3/cos(\y)},{3*tan(\y)*sin(\t)})
arc (180-\y:180+\y:3cm);
}

\foreach \y in {37}
{
\draw[domain=\y-90:90-\y, variable=\t, color=black] plot ({3*tan(\y)*cos(\t) - 3/cos(\y)},{3*tan(\y)*sin(\t)});
}

\foreach \y in {60}
{
\draw[domain=\y-90:90-\y, variable=\t, color=black] plot ({3*tan(\y)*cos(\t) - 3/cos(\y)},{3*tan(\y)*sin(\t)});
}

\foreach \y in {6}
{
\pgfmathatan{3/(2*\y - 3)};
\draw[domain=315-\pgfmathresult:315+\pgfmathresult/3, variable=\t, color=black] plot ({sqrt(2*\y*\y-6*\y+9)*cos(\t) - \y},{sqrt(2*\y*\y-6*\y+9)*sin(\t) + \y});
}

\foreach \y in {6}
{
\pgfmathatan{3/(2*\y - 3)};
\draw[domain=45-\pgfmathresult/3:45+\pgfmathresult, variable=\t, color=black] plot ({sqrt(2*\y*\y-6*\y+9)*cos(\t) - \y},{sqrt(2*\y*\y-6*\y+9)*sin(\t) - \y});
}

\node at (-3.3, 0) {$p$};

\node at (-0.7, 1.2) {$c_1$};
\node at (-0.7, -1.2) {$c_2$};

\node at (-0.5, 0.1) {$W$};

\node at (-1.2, -0.3) {$W_0$};

\node at (-2.3, 0.0) {$\tilde{W}$};

\end{tikzpicture}
\end{center}
\centerline{\null \hspace{1.0em} Fig. 3}

\
We claim that 
$$m < K + \varepsilon - h_0 + h_1 < K$$ in $\tilde{W}$.

\

Indeed: Let $x \in \tilde{W}$ and assume first that $x\in V.$ As observed above, there is some $z \in \partial (V_1 \cup V_2)$, say $z \in \partial V_1$, s.t.
$$x \in B_{\rho}(z)$$ 
and  there is $y\in V_1$ s.t. 
$$ B_1(y) \cap \overline{W \cap V} = \{z\}.$$
Therefore
$$ \operatorname*{dist}(x,y) < \rho +1.$$
Using triangular inequality and that $ \operatorname*{dist}(\partial W_0, \tilde{W}) > 3\rho + 2$, we have  
$$ B_{2\rho + 1}(y) \subset B_{3\rho + 2}(x) \subset W_0.$$
Let $w_y$ be the solution associated to the annulus $B_{2\rho + 1}(y) \backslash B_{1}(y)$ given by Lemma \ref{lem}.
Define 
$$ w = w_y+ K + \varepsilon - h_0 $$
Then, using that $B_1(y) \subset V_1$, 
$$ w = \delta + K + \varepsilon - h_0 > K +\delta+ \varepsilon - \frac{K}{2} -\frac{\delta}{2}> \frac{K}{2}+\frac{\delta}{2} > m \quad {\rm on } \quad \partial B_1(y) $$
and, from $B_{2\rho +1}(y) \subset W_0$,
$$ w = h_0 + K + \varepsilon - h_0 = K +\varepsilon > m  \quad {\rm on } \quad \partial B_{2\rho +1}(y).$$
From the comparison principle,
$$ m < w \quad \rm in \quad B_{2\rho + 1}(y) \backslash B_{1}(y)$$
and, therefore
$$ m < w_y + K + \varepsilon - h_0 < h_1 + K + \varepsilon - h_0 \quad {\rm in } \quad B_{\rho + 1}(y) \backslash B_{1}(y).$$
Since $\operatorname*{dist}(x,y) < \rho +1$, then $x \in B_{\rho +1}(y)$. Hence, using that $x \not\in V_1 \cup V_2$, we have $x \in B_{\rho + 1}(y) \backslash B_{1}(y)$. In this case, $m(x) < h_1 + K + \varepsilon - h_0.$ Finally,  if $x\in V_1 \cup V_2$, the definition of $\varepsilon$ implies that $m(x) < K/2+\delta/2 \le K + \varepsilon - h_0 + h_1$
proving the claim.
\

To conclude the proof of the theorem,
 note that $\nu := -\varepsilon + h_0 -h_1 > 0$, since $\varepsilon < h_0 - h_1$. Then $$K + \varepsilon - h_0 + h_1 = K - \nu$$
and, from the above claim,
$$ m < K -\nu < K \quad {\rm in } \quad \tilde{W}.$$
Hence
$$ \limsup_{x \to p} m(x) \le K -\nu < K$$
leading a contradiction.

\subsection{Proof of Theorem \ref{sec}.}
\begin{proof}
The proof that $m$ is bounded follows the same idea as in Theorem \ref{fi} replacing the geodesics $\Gamma_i$ and $\beta_k$ by totally geodesic hyperspheres $H_i$ and $\Lambda_k$ respectively and considering the same $S$. To build a supersolution $w_k$ such that $w_k=+\infty$ on $\Lambda_k$, we use the same construction as in Lemma \ref{existenceOfSomeScherk}, that is, we consider
$$ g(d)= S + \int_d^{\infty} \Aop^{-1} \left( \frac{K_0 }{(\cosh (a t ))^{n-1}}\right) \; dt, $$
that is well defined and finite for all $d >0$. The function $w_k(x):=g(d(x))$, where $d(x)=dist(x,\Lambda_k)$, is a supersolution according to \cite{RT2}. Moreover it satisfies $w_k(x)=+\infty$ for $x\in \Lambda_k$ since $g(0)=+\infty$ as a result of \eqref{IntegralCondition1Ona(t)}. Using this $w_k$, we conclude in the same way as in Theorem \ref{fi} that $m$ is bounded from above by $S$. In the same way, $m$ is bounded from below.

Now we prove that $m$ is continuous at $p \in \{p_1, \dots, p_k\}$. Denote by $\phi$ the continuous extension of  $m|_{\partial _{\infty }M\setminus\left\{ p_{1},...,p_{k}\right\} }$ to $\partial _{\infty }M.$
  Adding a constant to $\phi$ we may assume  wlg that $\phi(p)=0.$

 Hence we have to prove that 
$$ \lim_{x \to p} m(x) = 0.$$
Let 
$$ K = \limsup_{x\to p} m(x). $$
We will show that, for any $\delta >0$, it follows that $K \le \delta$. Since $v \le S$, it follows that  $K \le S$. Suppose that $K > \delta$. Let $V_j$ be a decreasing sequence of neighborhood of $p$ such that $$\bigcap \overline{V}_j = \{p\} \quad {\rm , }  \quad \sup_{x \in V_j} m(x) < K + 1/j \quad {\rm and} \quad \phi \le \frac{\delta}{2}  \quad {\rm on }\quad \partial_{\infty}V_j $$ 
We can suppose that each $V_j$ is a totally geodesic hyperball centered at $p$. (By a totally geodesic hyperball of $\mathbb{H}^n$ we mean a domain in  $ \mathbb{H}^n$ whose boundary is a totally geodesic hypersurface of $ \mathbb{H}^n.$)

\begin{center}
\begin{tikzpicture}[domain=-4:4]
\draw (0,0) circle  (3);
\foreach \y in {25}
{
\draw[domain=\y-90:90-\y, variable=\t, color=black] plot ({3*tan(\y)*cos(\t) - 3/cos(\y)},{3*tan(\y)*sin(\t)});
}

\foreach \y in {50}
{
\draw[domain=\y-90:90-\y, variable=\t, color=black] plot ({3*tan(\y)*cos(\t) - 3/cos(\y)},{3*tan(\y)*sin(\t)});
}

\draw (0,-3)--(0,2.40);
\draw (0,2.9)--(0,3);

\foreach \y in {150}
{
\draw[domain=\y-90:90-\y, variable=\t, color=black] plot ({3*tan(\y)*cos(\t) - 3/cos(\y)},{3*tan(\y)*sin(\t)});
}

\foreach \y in {15}
{
\draw[domain=180+\y:360-\y, variable=\t, color=black] plot ({3*tan(\y)*cos(\t) },{3*tan(\y)*sin(\t) +  3/cos(\y)});
}

\foreach \y in {15}
{
\draw[domain=\y:180-\y, variable=\t, color=black] plot ({3*tan(\y)*cos(\t) },{3*tan(\y)*sin(\t) -  3/cos(\y)});
}

\filldraw[black] (-3,0) circle (2pt);
\node at (-3.3, 0) {$p$};
\node at (-2.5, 0.5) {$\tilde{V}_j$};
\node at (-1.6, 0.8) {$V_j$};
\node at (-0.3, 0.3) {$A$};
\node at (1.5, 1.5) {$T_j(V_j)$};
\node at (0,2.6) {$B_q$};
\node at (-0.2,3.3) {$q$};
\filldraw[black] (0,3) circle (2pt);
\node at (-2.5, -0.7) {$x_j$};
\filldraw[black] (-2.3, -0.45) circle (1.5pt);
\node at (-0.4, -1.7) {$y_j$};
\filldraw[black] (0, -1.7) circle (1.5pt);
\end{tikzpicture}
\end{center}
\centerline{\null \hspace{1.0em} Fig. 4}

\

For each $j$, let $\tilde{V}_j \subset V_j$ be a totally geodesic hyperball centered at $p$ such that  
$$ dist( \partial \tilde{V}_j, \partial V_j ) \ge j  \quad {\rm and} \quad \sup_{x \in \tilde{V}_j} m(x) > K -1/j.$$
Then there exists a sequence $(x_j)$ that satisfies $x_j \in \tilde{V}_j$ and
$$ K - 1/j < m(x_j) < K + 1/j .$$ 
Denote $A = V_1$.  
It is well known that there exists an isometry  $T_j:\mathbb{H}^n \to \mathbb{H}^n$  that preserves $p$,  $T_j(\tilde{V}_j) \supset A$ and $y_j:=T_j(x_j) \in \partial A.$ We can suppose that $T_j(V_j)$ is an increasing sequence and that $\partial_{\infty}A \subset {\rm int} \; \partial_{\infty} T_j(V_j)$ for any $j$.
Observe that $$u_j= m \circ T^{-1}_j$$
is a solution of \eqref{generalOperator} and satisfies 
\begin{equation}
  \sup_{T_j(V_j)}u_j < K + 1/j \quad {\rm and } \quad u_j(y_j) > K - 1/j.
\label{limiteSuperiorWm}
\end{equation}
Moreover  $\tilde{V}_j \subset V_j \subset A \subset T_j(\tilde{V}_j)$ implies that 
\begin{align*}
 dist( \partial T_j (V_j) ,A ) &\ge dist (\partial T_j(V_j), T_j(\tilde{V}_j)) \\[5pt]
                             &= dist(\partial V_j, \tilde{V}_j)\ge j \to \infty .
\end{align*}

Observe that $T_j(V_j)$ is a totally geodesic hyperball and 
$$u_j \le \frac{\delta}{2} \quad {\rm  on} \quad  \partial_{\infty} (T_j(V_j)) \backslash \{p\} ,$$
since $u_j = m \circ T^{-1}_j$  and $m=\phi \le \delta/2$ on $V_j \backslash \{p\}$. Using that $A \subset T_j(V_j)$ and $p \not\in \partial_{\infty}( \mathbb{H}^n \backslash A )$, we have that $\partial_{\infty}A \cap \partial_{\infty}( \mathbb{H}^n \backslash A )\subset \partial_{\infty}T_j(V_j)\backslash \{p\}$ and, therefore, $u_j \le \delta/2$ on $\partial_{\infty}A \cap \partial_{\infty}( \mathbb{H}^n \backslash A)$. For $q \in \partial_{\infty}A \cap \partial_{\infty}( \mathbb{H}^n \backslash A)$, let $B_q$ be a totally geodesic hyperball centered at $q$ disjoint with $V_2$ such that $B_q \subset T_j(V_j)$ for any $j$. (This is possible since $(V_j)$ is a decreasing sequence, $T_j(V_j)$ is an increasing sequence and $\partial_{\infty}A \subset {\rm int} \, \partial_{\infty} T_j(V_j)$). In the same way as we did in the beginning, we can find supersolutions $w_q$ of
$$ \left\{ \begin{array}{rcl}   {\rm div} \left(\displaystyle \frac{\Aop (|\nabla u|)}{|\nabla u |} \nabla u \right) & = & 0 \quad {\rm in } \quad B_q \\[5pt]
                                                                                u   & = & +\infty \quad {\rm on } \quad  \partial B_q \\[5pt]
                                                                                u   & = & \delta/2 \quad {\rm on } \quad {\rm int} \; \partial_{\infty} B_q. \\ \end{array} \right. $$
Since $u_j \le w_q=\delta/2$ on ${\rm int} \; \partial_{\infty} B_q$, the comparison principle implies that $u_j \le w_q$ in $B_q$. Let $\tilde{B}_q \subset B_q$ be the hyperball with boundary equidistant to $\partial B_q$, for which $w_q < \delta$ in $\tilde{B}_q$. Hence $u_j < \delta$ in $\tilde{B}_q$ and, therefore, $u_j < \delta$ in $\tilde{B}$ for any $j$, where
$$ \tilde{B} = \bigcup_{q \in \partial_{\infty}A \cap \partial_{\infty}( \mathbb{H}^n \backslash A)} \tilde{B}_q .$$
Observe that $\tilde{B}$ is a neighborhood of $\partial_{\infty}A \cap \partial_{\infty}( \mathbb{H}^n \backslash A)$ and $\partial A \backslash \tilde{B}$ is compact.
\\ \\ \indent Now we prove that there exist $\nu > 0$ and $j_0 \in \mathbb{N}$ such that $u_j(y) \le K - \nu$ for any $j \ge j_0$ and $y \in \partial A$ contradicting $u_j(y_j) > K -1/j$ and $y_j \in \partial A$. 
\\ \indent Let $y$ be some point of $\tilde{B}$ such that the ball of radius $1$ centered at $y$, $B_1(y)$, is contained in $\tilde{B}$. Due to the fact that $\partial A \backslash \tilde{B}$ is compact, there exist $\rho >0$ such that the ball of radius $\rho +1$, $B_{\rho+1}(y)$, contain $\partial A \backslash \tilde{B}$.  Henceforth, we proceed as in Theorem \ref{fi}, using Lemma \ref{lem}. This lemma also holds in $\mathbb{H}^n$ and to prove it we define $f:[1,\infty)\to \mathbb{R}$ by
$$ f(r) = \delta + \int_1^r \Aop^{-1}\left( \frac{\sinh^{n-1} (\alpha)}{\sinh^{n-1}(s)} \right)\; ds \quad {\rm with } \quad 0< \alpha \le 1,$$
that satisfies
$$\Aop'(f'(r))f''(r)+ \Aop(f'(r))(n-1)\coth  r = 0, $$
and apply the same argument, obtaining a supersolution (indeed a solution) $w_y(x)=f(d(x))$. Then, we can consider $h_0$ and $h_1$ as in Lemma \ref{lem} and define
$w=w_y + K + \varepsilon - h_0$, where $\varepsilon$ satisfies $h_0-h_1 - (K-\delta)/2 \le \varepsilon < h_0 - h_1$. Take $j_0$ such that $1/j_0 < \varepsilon$. From \eqref{limiteSuperiorWm},
$$ \sup_{\partial A} u_j \le \sup_{T_j(V_j)} u_j < K + 1/j  < K + \varepsilon \quad {\rm for } \quad j \ge j_0.$$
Hence, following the same computation as in Theorem \ref{fi}, $w$ is a supersolution that satisfies $ w \ge u_j$ in $B_{2\rho + 1}(y) \backslash \overline{B_1(y)}$ for any $j \ge j_0$. Moreover $w <h_1 + K + \varepsilon - h_0$ in $B_{\rho + 1}(y) \backslash B_{1}(y) \supset \partial A \backslash \tilde{B}$. In $\partial A \cap \tilde{B}$, we also have $u_j < \delta <  h_1 + K + \varepsilon - h_0$. Thus, defining $\nu = h_0 - h_1 - \varepsilon > 0$, it follows that
$$ u_j  < K - \nu \quad {\rm in } \quad \partial A \quad {\rm for } \quad j \ge j_0.$$ 
But this contradicts $u_j(y_j) > K -1/j$ for any $j$. Therefore $K=0$. In a similar way $\liminf_{x \to p} m(x) \ge 0$ completing the proof.
 
\end{proof}

\subsection{Proof of Theorem \ref{nonexistenceThm} }
\begin{proof}
The idea is to build solutions that are constant along horospheres for which the asymptotic boundary is $p_1$. For that, let $H_1$ be some horosphere such that the asymptotic boundary is $p_1$ and $d(x)$ the distance with sign given by
$$ d(x) = \left\{ \begin{array}{rl}  dist(x, \partial H_1) & {\rm if } \; x \in H_1 \\[5pt]
                            - dist(x, \partial H_1) & {\rm if } \; x \not\in H_1 .\end{array} \right. $$
We search solutions of the form $m(x)=g(d(x))$, where $g:\mathbb{R}\to \mathbb{R}$ is a positive increasing function.
From \eqref{generalOperator}, we have that $g$ satisfies
$$ \Aop'(g'(d)) g''(d) + \Aop(g'(d)) \Delta d =0.$$
Since $d(x)$ is the distance (with sign) between $x$ and the horosphere $H_1$, then $\Delta d(x)= -(n-1)$. Therefore
\begin{equation}
\Aop'(g'(d)) g''(d) -(n-1) \Aop(g'(d)) =0.
\label{relationForG(d)}
\end{equation}
To find a solution to this equation, note first that $\Aop^{-1}(t)$ is defined for any $t >0$, since $\Aop$ is unbounded. Hence we can consider the function
$$ g_0(d)=\int_{-\infty}^d \Aop^{-1}(e^{(n-1)s}) \; ds $$
for all $d \in \mathbb{R}$. This integral converges at $-\infty$ since condition \eqref{basicHypothesesOna(s)} implies that $\Aop^{-1}(t) \le (t/\bar{D})^{1/q}$ for $\Aop^{-1}(t) \in [0,\delta_0]$. Observe that $g_0$ is positive, increasing, satisfies equation \eqref{relationForG(d)}, converges to $0$ as $d\to -\infty$ and diverges to $+\infty$ as $d\to +\infty$, because $\Aop^{-1}$ is increasing. Therefore
$$m(x)=g_0(d(x))$$ is a solution of \eqref{generalOperator} that satisfies $m(x_k)\to +\infty$ if $x_k \to p_1$ with $d(x_k) \to +\infty$.
Moreover, using that $d(x)\to -\infty$ as $x \to p \in \partial_{\infty}\mathbb{H}^n \backslash \{p_1\}$, it follows that $m(x) \to 0$ proving the result.
\end{proof}

\

\centerline{
\begin{tabular}{ccc}
{\small Leonardo Bonnorino} & $\hspace{1.5cm}$& {\small Jaime  Ripoll} \\
{\small UFRGS}  &  & {\small UFRGS}\\
{\small Instituto de Matem\'atica} &
& {\small Instituto de Matem\'atica}\\
{\small Av. Bento Gon\c calves 9500} & & {\small Av. Bento Gon\c calves 9500}\\
{\small 91540-000 Porto Alegre-RS } & &{\small 91540-000 Porto Alegre-RS }\\
{\small  BRASIL} &  & {\small BRASIL} \\
{\small leonardo.bonorino@ufrgs.br}& &{\small jaime.ripoll@ufrgs.br} \\
\end{tabular}}

\end{document}